\documentclass[11pt]{article}
\usepackage[utf8]{inputenc}
\usepackage{amssymb,amsmath,amsthm}

\usepackage[margin=2cm]{geometry}

\usepackage[backend=bibtex,style=numeric]{biblatex}
\newtheorem{theorem}{Theorem}[section]

\newtheorem{lemma}[theorem]{Lemma}

\theoremstyle{definition}
\newtheorem{definition}[theorem]{Definition}
\newtheorem{remark}[theorem]{Remark}
\usepackage{xcolor,hyperref}
\definecolor{darkblue}{rgb}{0,0,0.6}
\usepackage{xcolor,hyperref}
\hypersetup{
    colorlinks=true,       % false: boxed links; true: colored links
    linkcolor=darkblue,          % color of internal links (red)
    citecolor=darkblue,        % color of links to bibliography (green)
    filecolor=darkblue,      % color of file links (magenta)
    urlcolor=darkblue           % color of external links (cyan)
}

\usepackage[colorinlistoftodos,prependcaption,textsize=tiny]{todonotes}

\usepackage{mathtools,nicematrix_neu,stmaryrd,euscript}

\usepackage{amsopn}
\newcommand{\R}{\mathbb{R}}
\newcommand{\N}{\mathbb{N}}
\newcommand{\G}{G^{(n)}(x,y)}

\renewcommand{\div}{\operatorname{div}}
\DeclareMathOperator{\curl}{curl}
\DeclareMathOperator{\D}{D}
\DeclareMathOperator{\Div}{div}
\DeclareMathOperator{\inc}{\boldsymbol{\operatorname{inc}}}
\newcommand{\skalarProd}[2]{\left\langle#1,#2\right\rangle}
\newcommand{\norm}[1]{\left\lVert#1\right\rVert}
\newcommand{\Sym}[1]{\operatorname{Sym}\!\left(#1\right)}
\newcommand{\so}[1]{\mathfrak{so}\!\left(#1\right)}
\newcommand{\Anti}[1]{\operatorname{Anti}\!\left(#1\right)}
\newcommand{\axl}{\operatorname{axl}}
\newcommand{\ntimes}[1]{\times_{#1}}
\DeclareMathOperator{\Curl}{curl}%%groß Curl für distributional formulas
\newcommand{\nCurl}[1]{\Curl_{#1}}
\newcommand{\ncurl}[1]{\operatorname{curl}_{#1}\hspace{-1pt}}
\newcommand{\A}{\mathfrak{A}}
\renewcommand{\a}{\mathfrak{a}}
\newcommand{\nMat}[2]{\left\llbracket#1\right\rrbracket_{\ntimes{#2}}}
\newcommand{\sym}{\operatorname{sym}}
\renewcommand{\skew}{\operatorname{skew}}
\newcommand{\tr}{\operatorname{tr}}
\newcommand{\bbone}{\text{\usefont{U}{bbold}{m}{n}1}}
\MakeRobust{\bbone} 
\newcommand{\id}[1]{I_{#1}}
\renewcommand{\P}{\mathfrak{P}}

\title{Matrix representation of a cross product and related curl-based differential operators in all space dimensions}
\author{Peter Lewintan\thanks{Faculty of Mathematics, University of Duisburg-Essen, Thea-Leymann-Str. 9,
45127 Essen, Germany
  (\href{mailto:peter.lewintan@uni-due.de}{peter.lewintan@uni-due.de}, \url{udue.de/lew}).}
}

\usepackage{tikz}

\begin{document}
\maketitle
 
\begin{tikzpicture}[remember picture, overlay]
 \node [xshift=-1cm,yshift=15cm,rotate=-90] at (current page.south east)
 {The final publication appeared in Open Mathematics (2021), doi: \href{https://doi.org/10.1515/math-2021-0115}{10.1515/math-2021-0115}.
 };
\end{tikzpicture}
\numberwithin{equation}{section}

\begin{abstract}
A higher dimensional generalization of the cross product is associated with an adequate matrix multiplication. This index-free view allows for a better understanding of the underlying algebraic structures, among which are generalizations of Grassmann's, Jacobi's and Room's identities. Moreover, such a view provides a higher dimensional analogue of the decomposition of the vector Laplacian which itself gives an explicit index-free Helmholtz decomposition in arbitrary dimensions $n\ge2$.
\end{abstract}

\par

\par

\section{Introduction}
The interplay between different differential operators is at the basis not only of pure analysis but also of many applied mathematical considerations. One possibility is to study, instead of the properties of a linear homogeneous differential operator with constant coefficients 
\begin{subequations}
\begin{equation}
 \mathcal{A}=\sum_{|\boldsymbol{\alpha}|=k}A_{\boldsymbol\alpha}\,\nabla^{\boldsymbol\alpha}
\end{equation}
where $\boldsymbol\alpha=(\alpha_1,\ldots,\alpha_n)^T\in\N_0^n$ is a multi-index of length $|\boldsymbol\alpha|\coloneqq \alpha_1+\ldots+\alpha_n$,\linebreak $\nabla^{\boldsymbol\alpha}\coloneqq\partial_1^{\alpha_1}\ldots\partial_n^{\alpha_n}$ and $A_{\boldsymbol\alpha}\in\R^{N\times m}$, its symbol
\begin{equation}
 \mathbb{A}(b)=\sum_{|\boldsymbol{\alpha}|=k}A_{\boldsymbol\alpha}\,b^{\boldsymbol\alpha} \in \R^{N\times m},
\end{equation}
\end{subequations}
where we use the notation $b^{\boldsymbol\alpha}=b_1^{\alpha_1}\cdot\ldots\cdot b_n^{\alpha_n}$ for $b\in\R^n$. Note that $\mathcal{A}:C^\infty_c(\Omega,\R^m)\to C^\infty_c(\Omega,\R^N)$ with $\Omega\subseteq\R^n$ open and we obtain for all $a\in C^\infty_c(\Omega,\R^m)$ also the expression $\mathcal{A}\,a = \EuScript{A}(\D a)$ with $\EuScript{A}\in\operatorname{Lin}(\R^{m\times n},\R^N)$. The approach to look and algebraically operate with the vector differential operator $\nabla$ 
in a manner of a vector is also referred to as vector calculus or formal calculations.

An example of such a differential operator is the derivative $\D$ itself, but also $\div$, $\curl$, $\Delta$ or $\inc$. One of the most prominent relation in vector calculus is ~$\curl \nabla\zeta\equiv0$~ for scalar fields $\zeta\in C^\infty_c(\Omega)$, $\Omega\subseteq\R^3$ open, which, from an algebraic point of view, reads $b \times b = 0$ for all $b\in\R^3$ (where a scalar factor can be and is omitted).

In this paper we take a closer look at a higher dimensional analogue of the curl or rather the underlying generalized cross product. An extension of the usual cross product of vectors in $\R^3$ to vectors in $\R^n$  depends on which properties are to be fulfilled. The three basic properties of the vector product are: the linearity in both arguments, that the vector $a\times b$ is perpendicular to both $a,b\in\R^3$ (and, thus belongs to the \textit{same} space) and that its length is the area of the parallelogram spanned by $a$ and $b$. Gibbs uses these properties also to define the cross product, see \cite[Chapter II]{Gibbs1881Elements}. It turns out that such a vector product exists only in three and seven dimensions, cf.~\cite{Massey1983}. However, the $7$-dimensional vector product does not satisfy Jacobi's identity but rather a generalization of it, namely the \textit{Malcev identity}, cf.~\cite[p.~279]{Ebbinghaus-Numbers} and the references contained at the end of the section therein. We do not follow these constructions here and instead generalize the cross product to all dimensions by omitting one of its basic properties. These considerations are usually carried out using coordinates, i.e.~index notations. However, we are concerned with their matrix representation, which provides a better understanding of the underlying algebraic structures. Such a view has already proved very useful in extending Korn inequalities for incompatible tensor fields to higher dimensions, cf.~\cite{agn_lewintan2020KornLpN_tracefree}, where first results in these matrix representations have been obtained. In the present paper, we catch up with the underlying algebraic structures, among which are generalizations of Grassmann's, Jacobi's and Room's identities. Moreover, such a view provides a higher dimensional analogue of the decomposition of the vector Laplacian which itself gives an explicit index-free Helmholtz decomposition in arbitrary dimensions $n\ge2$.

\section{Notations}
As usual, $.\otimes.$ and $\skalarProd{.}{.}$ denote the dyadic and the scalar product, respectively. We write $.\cdot.$ to highlight the scalar multiplication of a scalar with a vector or a matrix.
The space of symmetric $(n\times n)$-matrices is denoted by $\Sym{n}$ and the space of skew-symmetric  $(n\times n)$-matrices by $\so{n}$. We use lower-case Greek letters to denote scalars, lower-case Latin letters to denote column vectors and upper-case Latin letters to denote matrices, with the exceptions for the dimensions: if not otherwise stated we have $n,m,N\in\N$ and $n\ge2$. The identity matrix is denoted by $\id{n}$. For the symmetric part, the skew-symmetric part and the transpose of a matrix $P$ we write $\sym P$, $\skew P$ and $P^T$, respectively.

\section{Algebraic view of a generalized cross product}
\subsection{Inductive introduction}
From an algebraic point of view the components of the usual cross product $a\times b$ are of the form $\alpha_i\beta_j-\alpha_j\beta_i$ for $1\le i<j\le 3$ sorted (and multiplied with $-1$) in such a way that the resulting vector is perpendicular to both $a$ and $b$. For a general $n\in\N$ we have $\frac{n(n-1)}{2}$ combinations of the form $\alpha_i\beta_j-\alpha_j\beta_i$ with $1\le i<j\le n$ and we array them as the column vector
\begin{equation}\label{eq:crossCoords}
\begin{pmatrix}
  \alpha_1\,\beta_2-\alpha_2\,\beta_1\\
  \alpha_1\,\beta_3-\alpha_3\,\beta_1\\
  \alpha_2\,\beta_3-\alpha_3\,\beta_2\\
  \alpha_1\,\beta_4-\alpha_4\,\beta_1\\
  \alpha_2\,\beta_4-\alpha_4\,\beta_2\\
  \alpha_3\,\beta_4-\alpha_4\,\beta_3\\
  \vdots
 \end{pmatrix}
 \qquad \text{for $a=(\alpha_i)_{i=1,\ldots,n}, b=(\beta_i)_{i=1,\ldots,n}\in\R^n$.}
\end{equation}
This becomes a vector from $\R^{\frac{n(n-1)}{2}}$ and only for $n=3$ lies in the same space as the vectors $a$ and $b$. More precisely, using the notation
\begin{equation}
 b =(\overline{b},\beta_n)^T\in\R^n \quad \text{with $\overline{b}\in\R^{n-1}$}
\end{equation}
we introduce the following generalized cross product $\ntimes{n}:\R^n\times\R^n\to\R^{\frac{n(n-1)}{2}}$ inductively by
\begin{equation}\label{eq:iductivecross}
 a\ntimes{n}b \coloneqq \begin{pmatrix} \overline{a}\ntimes{n-1}\overline{b} \\[1ex]
                \beta_n\cdot\overline{a}-\alpha_n\cdot\overline{b}
                \end{pmatrix}\in\R^{\frac{n(n-1)}{2}}
                \quad \text{where}\quad
\begin{pmatrix}\alpha_1\\ \alpha_2 \end{pmatrix}\ntimes{2}\begin{pmatrix}\beta_1\\\beta_2 \end{pmatrix}
\coloneqq \alpha_1\,\beta_2-\alpha_2\,\beta_1,
\end{equation}
wherefrom the bilinearity and anti-commutativity follow immediately. We show in section \ref{sec:Lagrange} that this generalized cross product $\ntimes{n}$ also satisfies the area property:
\begin{equation*}
 \norm{a\ntimes{n}b}^2=\norm{a}^2\norm{b}^2-\skalarProd{a}{b}^2 \quad \forall\ a,b\in\R^n.
\end{equation*}

\begin{remark}
 The anti-commutativity of the (usual or generalized) cross product is a consequence of the area property. Indeed, let $n,d\in\N$, $n\ge2$ and $\,\underline{\times}\,:\R^n\times \R^n \to \R^d$ be a bilinear map which satisfies the area property
 \begin{equation}\label{eq:areaGen}
  \norm{a \,\underline{\times}\,b }^2 = \norm{a}^2\norm{b}^2-\skalarProd{a}{b}^2 \qquad \forall\ a,b\in\R^n.
 \end{equation}
Then for $a=b$ we obtain:
\begin{equation}
 \norm{b\,\underline{\times}\,b}^2 = \norm{b}^2\norm{b}^2-\skalarProd{b}{b}^2=0 \quad \Rightarrow \quad b\,\underline{\times}\,b = 0 \qquad \forall \ b \in\R^n.
\end{equation}
Linearizing the last equality leads to
\begin{align}
 0=(a+b)\,\underline{\times}\,(a+b)=a\,\underline{\times}\,a+a\,\underline{\times}\,b+b\,\underline{\times}\,a+b\,\underline{\times}\,b \quad \Leftrightarrow \quad a\,\underline{\times}\,b = - b\,\underline{\times}\,a \quad \forall \ a,b\in\R^n.
\end{align}
Furthermore, in case $d=n$, we call $\,\underline{\times}\,$ a \emph{vector product} to emphasize  that the vector $a\,\underline{\times}\,b$ is in the \textit{same} space as $a$ and $b$. In this situation we can further talk about orthogonality of the vector $a\,\underline{\times}\,b$ to both $a$ and $b$. Massey \cite{Massey1983} showed that assuming these three properties, i.e.~bilinearity, area property and orthogonality, a vector product exists only in the dimension $n=3$ or $n=7$. However, there are many cross products, different from each other, depending on the properties one requires to hold. In the present paper, we drop the orthogonality condition since we consider the case $d=\frac{n(n-1)}{2}$ and introduce in \eqref{eq:iductivecross} the generalized cross product by induction over the space dimension $n$. This is equivalent to the coordinate-wise expression from \eqref{eq:crossCoords} but allows for a better understanding of the algebraic properties of the generalized cross product $\ntimes{n}$.
\end{remark}

\subsection{Relation to skew-symmetric matrices}
To establish the connection of the generalized cross product $a\ntimes{n}b$ to the entries of $\skew(a\otimes b)$ we start with the following bijection $\a_n: \so{n}\to\R^{\frac{n(n-1)}{2}}$ given by
 \begin{subequations}
 \begin{equation}
  \a_n (A)\coloneqq(\alpha_{12},\alpha_{13},\alpha_{23},\ldots,\alpha_{1n},\ldots,\alpha_{(n-1)n})^T \quad 
 \end{equation}
for $A=(\alpha_{ij})_{i,j=1,\ldots,n}\in\so{n}$, as well as its inverse $\A_n:\R^{\frac{n(n-1)}{2}}\to\so{n}$, so that 
\begin{equation}
\begin{split}
  \A_n(\a_n (A)) &= A  \quad \forall\ A\in\so{n} \quad \text{and} \\
  \a_n(\A_n (\a)) &= \a  \quad \forall\ \a\in\R^{\frac{n(n-1)}{2}}
\end{split}
\end{equation}
\end{subequations}
and, for $\a=\left(\alpha_1,\ldots,\alpha_{\frac{n(n-1)}{2}}\right)^T\in\R^{\frac{n(n-1)}{2}}$, in coordinates it looks like 
\begin{equation}
 \A_n(\a)=
 \begin{pmatrix}
 0 & \alpha_1 & \alpha_2 & \alpha_4 & \vdots \\
-\alpha_1 & 0 & \alpha_3 & \alpha_5 & \vdots\\ 
-\alpha_2 & -\alpha_3 & 0 & \alpha_6& \vdots\\
-\alpha_4 & -\alpha_5& -\alpha_6 & 0 & \vdots\\
\cdots & \cdots & \cdots & \cdots & 0
\end{pmatrix}\,.
\end{equation}
Thus, the generalized cross product $a\ntimes{n}b$ can be written as
\begin{subequations}\label{eq:crossandskew}
\begin{align}
 a\ntimes{n}b &=\a_n(a\otimes b -b\otimes a)=2\cdot\a_n(\skew(a\otimes b)),
 \intertext{or, equivalently,}
 \A_n(a\ntimes{n}b) &= a\otimes b-b\otimes a  \qquad \text{is true for all $a,b\in\R^n$}.
\end{align}
\end{subequations}

\subsection{Lagrange identity}\label{sec:Lagrange}
In three dimensions, Lagrange's identity reads in terms of the usual cross product and the scalar product
\begin{equation}\label{eq:Lagrange3}
 \skalarProd{a\times b}{c\times d}=\skalarProd{a}{c}\skalarProd{b}{d}-\skalarProd{a}{d}\skalarProd{b}{c} \qquad \forall \ a,b,c,d\in\R^3
 \end{equation}
and for $c=a$ and $d=b$ becomes
\begin{equation}
 \norm{a\times b}^2=\norm{a}^2\norm{b}^2- \skalarProd{a}{b}^2 \qquad \forall \ a,b\in\R^3
\end{equation}
meaning that the length of the vector $a\times b \in\R^3$ is equal to the area of the parallelogram spanned by the vectors $a,b\in\R^3$.
 
In higher dimensions, the inductive definition \eqref{eq:iductivecross} can be used to directly deduce an analogue to \textit{Lagrange's identity}, namely:
\begin{equation}\label{eq:LagrangeN}
 \skalarProd{a\ntimes{n}b}{c\ntimes{n}d}=\skalarProd{a}{c}\skalarProd{b}{d}-\skalarProd{a}{d}\skalarProd{b}{c} \qquad \forall \ a,b,c,d\in\R^n.
\end{equation}
Indeed, in the dimension $n=2$ we have
\begin{align}
 \skalarProd{\begin{pmatrix}\alpha_1\\\alpha_2 \end{pmatrix}\ntimes{2}\begin{pmatrix}\beta_1\\\beta_2 \end{pmatrix}}{\begin{pmatrix}\gamma_1\\\gamma_2 \end{pmatrix}\ntimes{2}\begin{pmatrix}\delta_1\\\delta_2 \end{pmatrix}} 
  &= (\alpha_1\,\beta_2-\alpha_2\,\beta_1)(\gamma_1\,\delta_2-\gamma_2\,\delta_1) \notag\\
 & = \alpha_1\,\beta_2\,\gamma_1\,\delta_2+\alpha_2\,\beta_1\,\gamma_2\,\delta_1 
     -\alpha_1\,\beta_2\,\gamma_2\,\delta_1-\alpha_2\,\beta_1\,\gamma_1\,\delta_2 \notag\\
 & = (\alpha_1\,\gamma_1+\alpha_2\,\gamma_2)(\beta_1\,\delta_1+\beta_2\,\delta_2)-(\alpha_1\,\delta_1+\alpha_2\,\delta_2)(\beta_1\,\gamma_1+\beta_2\,\gamma_2)\notag\\
 & \hspace*{-3em}= \skalarProd{\begin{pmatrix}\alpha_1\\\alpha_2 \end{pmatrix}}{\begin{pmatrix}\gamma_1\\\gamma_2 \end{pmatrix}}\skalarProd{\begin{pmatrix}\beta_1\\\beta_2 \end{pmatrix}}{\begin{pmatrix}\delta_1\\\delta_2 \end{pmatrix}}-\skalarProd{\begin{pmatrix}\alpha_1\\\alpha_2 \end{pmatrix}}{\begin{pmatrix}\delta_1\\\delta_2 \end{pmatrix}}\skalarProd{\begin{pmatrix}\beta_1\\\beta_2  \end{pmatrix}}{\begin{pmatrix}\gamma_1\\\gamma_2 \end{pmatrix}}.\notag
\end{align}
Furthermore, with $a=(\overline{a},\alpha_n)^T, b=(\overline{b},\beta_n)^T,c=(\overline{c},\gamma_n)^T,d=(\overline{d},\delta_n)^T$ we obtain on the one hand
\begin{align}
  \skalarProd{a\ntimes{n}b}{c\ntimes{n}d}&=
  \skalarProd{\begin{pmatrix} \overline{a}\ntimes{n-1}\overline{b} \\[1ex]
                \beta_n\cdot\overline{a}-\alpha_n\cdot\overline{b}
                \end{pmatrix}}{\begin{pmatrix} \overline{c}\ntimes{n-1}\overline{d} \\[1ex]
                \delta_n\cdot\overline{c}-\gamma_n\cdot\overline{d}
                \end{pmatrix}}\notag\\
            & = \skalarProd{\overline{a}\ntimes{n-1}\overline{b}}{\overline{c}\ntimes{n-1}\overline{d}}+\beta_n\,\delta_n\skalarProd{\overline{a}}{\overline{c}}+\alpha_n\,\gamma_n\skalarProd{\overline{b}}{\overline{d}}-\beta_n\,\gamma_n\skalarProd{\overline{a}}{\overline{d}}-\alpha_n\,\delta_n\skalarProd{\overline{b}}{\overline{c}},\notag
\end{align}
and on the other hand:
\begin{align}
 \skalarProd{a}{c}\skalarProd{b}{d}-\skalarProd{a}{d}\skalarProd{b}{c}
 &= (\skalarProd{\overline{a}}{\overline{c}}+\alpha_n\,\gamma_n)
 (\skalarProd{\overline{b}}{\overline{d}}+\beta_n\,\delta_n)
 -
 (\skalarProd{\overline{a}}{\overline{d}}+\alpha_n\,\delta_n)
 (\skalarProd{\overline{b}}{\overline{c}}+\beta_n\,\gamma_n) \notag\\
 & = \skalarProd{\overline{a}}{\overline{c}}\skalarProd{\overline{b}}{\overline{d}}-\skalarProd{\overline{a}}{\overline{d}}\skalarProd{\overline{b}}{\overline{c}}
 +\beta_n\,\delta_n\skalarProd{\overline{a}}{\overline{c}}+\alpha_n\,\gamma_n\skalarProd{\overline{b}}{\overline{d}}-\beta_n\,\gamma_n\skalarProd{\overline{a}}{\overline{d}}-\alpha_n\,\delta_n\skalarProd{\overline{b}}{\overline{c}},\notag
\end{align}
so that \eqref{eq:LagrangeN} follows by induction over $n\in\N$ , $n\ge2$. Especially for $c=a$ and $d=b$ we obtain for the squared norm of the generalized cross product
\begin{equation}\label{eq:parallelogramN}
 \norm{a\ntimes{n}b}^2\overset{\eqref{eq:LagrangeN}}{=}\norm{a}^2\norm{b}^2-\skalarProd{a}{b}^2 \quad \forall\ a,b\in\R^n
\end{equation}
meaning that the length of the vector $a\ntimes{n}b\in\R^{\frac{n(n-1)}{2}}$ is equal to the area of the parallelogram spanned by the vectors $a,b\in\R^n$.

Two (non-zero) vectors $a,b\in\R^n$ are linearly dependent (and thus \emph{parallel}) if and only if $a\ntimes{n} b = 0$.

\subsection{Matrix representation}
It is well known that an identification of the usual cross product $\times$ with an adequate matrix multiplication facilitates some of the common proofs in vector algebra and allows one to extend the cross product of vectors to a cross product of a vector and a matrix, cf.~\cite{Room1952vectprod,GTT1999vectprodinC,Trenkler2001vectprod,TT2008vectorcross,agn_lewintan2020generalKorn}.

Our next goal is to achieve a similar identification of the generalized cross product $\ntimes{n}$ with a corresponding matrix multiplication. Indeed, since for a fixed $a\in\R^n$ the operation $a\ntimes{n}.$ is linear in the second component there exists a unique matrix denoted by $\nMat{a}{n}\in\R^{\frac{n(n-1)}{2}\times n}$ such that
\begin{equation}
 a\ntimes{n}b\eqqcolon \nMat{a}{n}b\qquad \forall \ b\in\R^n.
\end{equation}
In view of \eqref{eq:iductivecross} the matrices $\nMat{.}{n}$ can be characterized inductively, and for $a=(\overline{a},\alpha_n)^T$ the matrix $\nMat{a}{n}$ has the form
 \begin{equation}\label{eq:crossNmat}
  \nMat{a}{n}= \begin{pNiceArray}{CC:C}
  \Block{2-2}<\large>{\nMat{\overline{a}}{n-1}} & & \Block{2-1}{0} \\
  \hspace*{2em}&\hspace*{2em} & \hspace*{2em}  \\
  \Hdotsfor{3}\\
   \Block{2-2}{-\alpha_n\cdot\id{n-1}} & &  \Block{2-1}{\overline{a}}\\
   &&
   \end{pNiceArray}
\qquad \text{where}\quad \nMat{\begin{pmatrix}\alpha_1\\ \alpha_2 \end{pmatrix}}{2}= \begin{pNiceArray}{CC}-\alpha_2, & \alpha_1\end{pNiceArray},
 \end{equation}
so that 
\begin{equation}\label{eq:genMat3}
 \nMat{\begin{pmatrix}\alpha_1\\\alpha_2\\\alpha_3 \end{pmatrix}}{3}= 
 \begin{pNiceArray}{CC:C}
  -\alpha_2 & \alpha_1 & 0 \\
  \Hdotsfor{3}\\
  -\alpha_3 & 0& \alpha_1 \\
  0 & -\alpha_3 & \alpha_2
 \end{pNiceArray}
 \quad \text{etc.}
  \end{equation}
 \begin{remark}
 The entries of the generalized cross product $a\times_3 b$, with $a,b\in\R^3$, are permutations (with a sign) of the entries of the classical cross product $a\times b$. Remember that the operation $a\times.$ can be identified with the left multiplication by the following skew-symmetric matrix
 \begin{align}
  \operatorname{Anti}(a)&=\begin{pmatrix} 0 & -\alpha_3 & \alpha_2 \\ \alpha_3 & 0 & -\alpha_1 \\ -\alpha_2 & \alpha_1 & 0 \end{pmatrix}
\intertext{which differs from the expression $\nMat{a}{3}$ for $a=(\alpha_1, \alpha_2,\alpha_3)^T$, cf.~\eqref{eq:genMat3}, and also form $\A_3(a)$ which reads}
\A_3(a)&=\begin{pmatrix} 0 & \alpha_1 & \alpha_2 \\ -\alpha_1 & 0 & \alpha_3 \\ -\alpha_2 & -\alpha_3 & 0 \end{pmatrix}\,.
\end{align}
Thus, in three dimensions, it holds for the usual cross product
\begin{equation}
 a \times b = \Anti{a}b \qquad \forall\ a,b\in\R^3.
 \end{equation}
Also the notations $T_a$ , $W(a)$ or even $[a]_{\times}$ are used for $\Anti{a}$, however, the latter emphasizes that we deal with a skew-symmetric matrix. For the analysis and the properties of such matrices we refer to \cite{Room1952vectprod,GTT1999vectprodinC,Trenkler2001vectprod,TT2008vectorcross,agn_lewintan2020generalKorn}.
\end{remark}

\begin{remark}
Also the $7$-dimensional vector product $a\times.$ for $a\in\R^7$ (which differs from $a\ntimes{7}.$) can be represented with a left multiplication by a skew-symmetric matrix from $\so7$, see \cite{Leite1993cross7,LeiteCrouch1997cross7,BNAV2017cross7,BNSV2018vectorcross}.
\end{remark}

\subsection{Scalar triple product}
In case of the usual cross product in three dimensions the scalar triple product remains unchanged under a circular shift of the three vectors (from the same space):
\begin{equation}
 \skalarProd{a}{b\times c} = \skalarProd{a}{\Anti{b}c} = -\skalarProd{\Anti{b}a}{c}= \skalarProd{a\times b}{c}, \quad \forall\ a,b,c\in\R^3.
\end{equation}
Since $\ntimes{n}: \R^n\times \R^n\to \R^{\frac{n(n-1)}{2}}$ it does not make sense to think of an analogue of a scalar triple product with three vectors coming from the same vector space but rather instead:
\begin{subequations}\label{eq:scalartripleN}
\begin{align} 
 \skalarProd{\a}{b\ntimes{n}c} &= \skalarProd{\a}{\nMat{b}{n}c}=\skalarProd{\nMat{b}{n}^T\a}{c}
 \quad \forall\ \a\in\R^{\frac{n(n-1)}{2}}, \ b,c\in\R^n,
 \intertext{so that with $c=b$ we have:}\label{eq:scalartripleNbb}
 \skalarProd{\nMat{b}{n}^T\a}{b} &= 0 \quad \forall\ \a\in\R^{\frac{n(n-1)}{2}}, \ b\in\R^n.
\end{align}
\end{subequations}
Note the slight difference from the case of the usual cross product. The latter can be represented by a left multiplication with a square skew-symmetric matrix, whereas for the generalized cross product by matrices of the form \eqref{eq:crossNmat} which are neither square (except the case $n=3$) nor skew-symmetric matrices. These matrices $\nMat{.}{n}$ also appear in further identities involving the generalized cross product and are very important in the subsequent considerations.

\subsection{Grassmann identity} 
In three dimensions, the usual vector triple product fulfills
\begin{align}\label{eq:Grassmann3}
 a \times (b\times c)& =  \Anti{a}(b\times c) = \Anti{a}\Anti{b}c = -(b\times c) \times a = -\Anti{b\times c}a \notag \\
 &\overset{*}{=} \skalarProd{a}{c}\cdot b - \skalarProd{a}{b}\cdot c \qquad \forall \ a,b,c\in\R^3,
\end{align}
where the relation to scalar products, marked by $*$, is referred to as  \emph{Grassmann identity}.

However, in a generalization of a vector triple product we cannot expect the double appearance of the generalized cross product but focus on the matrices $\nMat{.}{n}$, as in the generalization of the scalar triple. Thus, as a generalization of \emph{Grassmann's identity}  we obtain for $a,b,c\in\R^n$ 
 \begin{align}\label{eq:GrassmannN}
  \nMat{a}{n}^T(b\ntimes{n}c)&=\skalarProd{a}{b}\cdot c-\skalarProd{a}{c}\cdot b = (\skalarProd{b}{a}\cdot\id{n}-b\otimes a)\,c\,\\
  &=(c\otimes b - b \otimes c)\, a \overset{\eqref{eq:crossandskew}}{=} -\A_n(b\ntimes{n}c)\,a \ \in\R^n \notag.
 \end{align}
It remains to prove the first equality \eqref{eq:GrassmannN}$_1$. In the dimension $n=2$ we have
\begin{align}
 \begin{pmatrix}-\alpha_2\\ \alpha_1 \end{pmatrix} \left(\begin{pmatrix}\beta_1\\ \beta_2 \end{pmatrix}\ntimes{2}\begin{pmatrix}\gamma_1\\ \gamma_2 \end{pmatrix}\right)
 & = \begin{pmatrix}-\alpha_2\\ \alpha_1 \end{pmatrix}\cdot(\beta_1\,\gamma_2-\beta_2\,\gamma_1)
 =  \begin{pmatrix}\alpha_2\,\beta_2\,\gamma_1 -\alpha_2\,\beta_1\,\gamma_2 \\
    \alpha_1\,\beta_1\,\gamma_2 -\alpha_1\,\beta_2\,\gamma_1 \end{pmatrix} \notag \\
    & = \skalarProd{\begin{pmatrix}\alpha_1\\ \alpha_2 \end{pmatrix}}{\begin{pmatrix}\beta_1\\ \beta_2 \end{pmatrix}}\cdot \begin{pmatrix}\gamma_1\\ \gamma_2 \end{pmatrix}- \skalarProd{\begin{pmatrix}\alpha_1\\ \alpha_2 \end{pmatrix}}{\begin{pmatrix}\gamma_1\\ \gamma_2 \end{pmatrix}}\cdot\begin{pmatrix}\beta_1\\ \beta_2 \end{pmatrix}.
\end{align}
Furthermore, with $a=(\overline{a},\alpha_n)^T, b=(\overline{b},\beta_n)^T,c=(\overline{c},\gamma_n)^T$ we obtain on the one hand
\begin{align}
 \nMat{a}{n}^T(b\ntimes{n}c)&=\begin{pNiceArray}{CC:C}
  \Block{2-2}<\large>{\nMat{\overline{a}}{n-1}^T} & & \Block{2-1}{-\alpha_n\cdot\id{n-1}} \\
  \hspace*{3em}&\hspace*{2em} & \hspace*{4em}  \\
  \Hdotsfor{3}\\
   \Block{2-2}{0} & &  \Block{2-1}{\overline{a}^T}\\
   &&
   \end{pNiceArray} 
   \begin{pmatrix} \overline{b}\ntimes{n-1}\overline{c} \\[1ex]
                \gamma_n\cdot\overline{b}-\beta_n\cdot\overline{c}
                \end{pmatrix}\notag\\
    &=
    \begin{pmatrix}
     \nMat{\overline{a}}{n-1}^T(\overline{b}\ntimes{n-1}\overline{c})+\alpha_n\,\beta_n\cdot\overline{c}-\alpha_n\,\gamma_n\cdot\overline{b}\\[1ex]
     \gamma_n\,\skalarProd{\overline{a}}{\overline{b}}-\beta_n\,\skalarProd{\overline{a}}{\overline{c}}
    \end{pmatrix}
\end{align}
and on the other hand:
\begin{align}
 \skalarProd{a}{b}\cdot c-\skalarProd{a}{c}\cdot b & = \skalarProd{\begin{pmatrix}\overline{a}\\ \alpha_n \end{pmatrix}}{\begin{pmatrix}\overline{b}\\ \beta_n \end{pmatrix}}\cdot \begin{pmatrix}\overline{c}\\ \gamma_n \end{pmatrix}- \skalarProd{\begin{pmatrix}\overline{a}\\ \alpha_n \end{pmatrix}}{\begin{pmatrix}\overline{c}\\ \gamma_n \end{pmatrix}}\cdot\begin{pmatrix}\overline{b}\\ \beta_n \end{pmatrix} \notag\\
 & =     \begin{pmatrix}\skalarProd{\overline{a}}{\overline{b}}\cdot \overline{c}-\skalarProd{\overline{a}}{\overline{c}}\cdot \overline{b}
   +\alpha_n\,\beta_n\cdot\overline{c}-\alpha_n\,\gamma_n\cdot\overline{b}\\[1ex]
    \skalarProd{\overline{a}}{\overline{b}}\, \gamma_n-\skalarProd{\overline{a}}{\overline{c}}\,\beta_n
    \end{pmatrix}
\end{align}
so that \eqref{eq:GrassmannN}$_1$ follows by induction over $n\in\N$ , $n\ge2$.

\subsection{Jacobi identity}
In three dimensions, the usual cross product satisfies the \emph{Jacobi identity}:
\begin{equation}\label{eq:Jacobi3}
 a\times(b\times c) + b \times (c\times a) + c \times (a\times b) = 0 \qquad \forall\ a,b,c\in\R^3,
\end{equation}
which follows directly from the usual Grassmann identity \eqref{eq:Grassmann3} for the usual vector triple product. Similarly, having established the generalization of Grassmann's identity involving the generalized cross product $\ntimes{n}$ in the previous section, we obtain the following generalization of \emph{Jacobi's identity}:
\begin{subequations}\label{eq:JacobiN}
\begin{align}
 \nMat{a}{n}^T(b\ntimes{n}c) + \nMat{b}{n}^T(c\ntimes{n}a) + \nMat{c}{n}^T(a\ntimes{n}b) &\overset{\eqref{eq:GrassmannN}_1}{=}0 \quad \forall\ a,b,c\in\R^n
 \intertext{or, equivalently:}
 \A_n(b\ntimes{n}c)\,a+\A_n(c\ntimes{n}a)\,b+\A_n(a\ntimes{n}b)\,c&\overset{\eqref{eq:GrassmannN}_4}{=}0.
\end{align}
\end{subequations}
Surely, the relation \eqref{eq:GrassmannN} can also be used to obtain \eqref{eq:LagrangeN}.

\subsection{Cross product with a matrix}
Furthermore, the generalized cross product can be written as
\begin{equation}
 a\ntimes{n}b=-b\ntimes{n}a=\nMat{-b}{n}a=\left(a^T\nMat{-b}{n}^T\right)^T\,.
\end{equation}
This allows us to define a generalized cross product of a vector $b\in\R^n$ and a matrix $P\in\R^{m\times n}$ from the right and with a matrix $B\in\R^{n\times m}$ from the left, where $m\in\N$, via
\begin{subequations}\label{eq:MatrCrossN}
 \begin{align}
  P\ntimes{n}b&\coloneqq P\nMat{-b}{n}^T\in\R^{m\times\frac{n(n-1)}{2}} \qquad \text{seen as row-wise cross product,}\label{eq:Matcrossprodright}
  \shortintertext{and}
  b\ntimes{n}B&\coloneqq\nMat{b}{n}B\in\R^{\frac{n(n-1)}{2}\times m} \qquad \text{seen as column-wise cross product,}
 \end{align}
and they are connected via
\begin{equation}\label{eq:crosslinkszurechts}
 (b\ntimes{n}B)^T = B^T\nMat{b}{n}^T = -B^T\ntimes{n}b \quad \forall\ B\in\R^{n\times m}, b\in\R^n.
\end{equation}
\end{subequations}
So, especially for the identity matrix $P=\id{n}$ we obtain
\begin{equation}
 \id{n}\ntimes{n}b = \nMat{-b}{n}^T \quad \text{and}\quad b\ntimes{n}\id{n}=\nMat{b}{n}.
\end{equation}
Moreover, for $a\in\R^m$ and $b,c\in\R^n$ it follows
\begin{subequations}
 \begin{align}
  (a\otimes b)\ntimes{n}c &= a\,b^T\nMat{-c}{n}^T = a\,(\nMat{-c}{n}b)^T = a\,(-c\ntimes{n}b)^T = a\otimes(b\ntimes{n} c),\label{eq:dyadiccrossN}
  \intertext{and especially for $c=b$:}
  (a\otimes b)\ntimes{n}b &= 0 \quad \text{for all $a\in\R^m$ and all $b\in\R^n$.}\label{eq:dyadiccross2N}
\intertext{As a consequence we obtain}
(b\otimes a)\ntimes{n}b\quad &\overset{\mathclap{\eqref{eq:dyadiccross2N}}}{=} \quad 2\cdot\sym(a\otimes b)\ntimes{n} b = -2\cdot\skew(a\otimes b)\ntimes{n}b \notag \\
&\overset{\mathclap{\eqref{eq:dyadiccrossN}}}{=} \quad b \otimes (a\ntimes{n}b)\overset{\eqref{eq:crossandskew}}{=}2\cdot b \otimes \a_n(\skew(a\otimes b)) \quad \text{for all $a,b\in\R^n$}.\label{eq:dyadiccross3N}
\end{align}
\end{subequations}

\subsection{Another vector triple}
Already in the scalar triple product we come across the expression $\nMat{b}{n}^T\a\in\R^n$. Hence, we may consider also the following vector triple product for $\a\in\R^{\frac{n(n-1)}{2}}$ and $b,c\in\R^n$:
\begin{align}\label{eq:vectripN}
 \left(\nMat{b}{n}^T\a\right)\ntimes{n}c& = \nMat{\nMat{b}{n}^T\a}{n}c\notag\\
 & = -c\ntimes{n}\nMat{b}{n}^T\a = -\nMat{c}{n}\nMat{b}{n}^T\a =\left(\nMat{c}{n}\ntimes{n} b\right)\a\ \in\R^{\frac{n(n-1)}{2}}.
\end{align}
Again, the corresponding relations to \eqref{eq:GrassmannN} and \eqref{eq:vectripN} for the usual cross product coincide, whereas the situation is different for the generalized cross product due to the non-symmetry of the matrices $\nMat{.}{n}$.

The inductive view \eqref{eq:crossNmat}$_1$ on the appearing matrix in \eqref{eq:vectripN} shows for all $a,b\in\R^n$:
\begin{align}\label{eq:acrossb_inRn}
 \nMat{a}{n}\ntimes{n}b&= \nMat{a}{n}\nMat{-b}{n}^T = 
 \notag \\[1ex]
 & = \begin{pNiceArray}{C:C}
     \nMat{\,\overline{a}\,}{n-1}\ntimes{n-1}\overline{b} & \beta_n\cdot\nMat{\overline{a}}{n-1}\\
     \Hdotsfor{2}\\
      \alpha_n\cdot\nMat{\,\overline{b}\,}{n-1}^T & -\overline{a}\otimes \overline{b}-\alpha_n\,\beta_n\cdot\id{n-1}
     \end{pNiceArray}\in\R^{\frac{n(n-1)}{2}\times\frac{n(n-1)}{2}},
\end{align}
and especially for $a=b$: 
\begin{equation}\label{eq:bcrossb_inRn}
 \nMat{b}{n}\ntimes{n}b =-\nMat{b}{n}\nMat{b}{n}^T \in\Sym{\frac{n(n-1)}{2}}.
\end{equation}

Moreover, we may also consider the following matrix multiplication:
\begin{equation}
 \P\nMat{b}{n}\in\R^{m\times n} \qquad \text{for} \ \P\in\R^{m\times\frac{n(n-1)}{2}}
\end{equation}
and, like in \eqref{eq:MatrCrossN}, related by transposition also $\nMat{b}{n}^T(.)$ for an $\left(\frac{n(n-1)}{2}\times m\right)$-matrix.

\subsection{Room identity}
Surely, the considerations in the previous subsections were inspired by the corresponding relations known for the usual cross product. So, from the usual Grassmann identity \eqref{eq:Grassmann3} one can deduce the usual Jacobi \eqref{eq:Jacobi3} and Lagrange \eqref{eq:Lagrange3} identities. Moreover, the usual Grassmann identity \eqref{eq:Grassmann3} for the vector triple in three dimensions allows also to conclude that 
\begin{equation}\label{eq:Room}
 \Anti{a}\Anti{b}=\Anti{a}\times b = b\otimes a -\skalarProd{a}{b}\cdot\id3 \quad \forall\ a,b\in\R^3.
\end{equation}
This algebraic relation is already contained in \cite[p. 691 (ii)]{Room1952vectprod}. For this reason let us call it \emph{Room identity}. The relation \eqref{eq:Room} turned out to be very important also from an application point of view, cf.~\cite{agn_lewintan2020generalKorn,agn_lewintan2020KornLp_tracefree} and the references contained therein.

Returning to the $n$-dimensional case, we have for arbitrary $a,b\in\R^n$:
\begin{equation}
 \nMat{a}{n}^T\nMat{b}{n}\,x=\nMat{a}{n}^T(b\ntimes{n}x)\overset{\eqref{eq:GrassmannN}}{=}(\skalarProd{b}{a}\cdot\id{n}-b\otimes a)\,x \qquad \forall\ x\in\R^n,
\end{equation}
so that as an analogue to \emph{Room's identity} it follows
\begin{equation}\label{eq:RoomN}
 \nMat{a}{n}^T\nMat{b}{n} = \skalarProd{b}{a}\cdot\id{n}-b\otimes a \in\R^{n\times n} \qquad \forall\ a,b\in\R^n,
\end{equation}
and especially for $a=b$:
\begin{equation}\label{eq:RoomNbb}
 \nMat{b}{n}^T\nMat{b}{n} = \norm{b}^2\cdot\id{n}-b\otimes b \in\Sym{n}.
\end{equation}
Note that the minus sign is missing in the generalized Room identity \eqref{eq:RoomN} due to the lack of skew-symmetry of the matrix $\nMat{a}{n}^T$.

Interchanging the roles of $a$ and $b$ in \eqref{eq:RoomN} we further deduce that
\begin{equation}
  \nMat{a}{n}^T\nMat{b}{n} -  \nMat{b}{n}^T\nMat{a}{n} = a\otimes b - b \otimes a \overset{\eqref{eq:crossandskew}}{=} \A_n(a\ntimes{n} b).
\end{equation}
Since $\tr(a\otimes b)=\skalarProd{a}{b}$, the expression \eqref{eq:RoomN} shows that the entries of $\nMat{a}{n}^T\nMat{b}{n}$ are linear combinations of the entries of the dyadic product $a\otimes b$. Also the converse holds true:
\begin{equation}\label{eq:RoomNdyadic0}
 b\otimes a = \frac{\nMat{a}{n}^T\nMat{b}{n}}{n-1}\cdot\id{n}-\nMat{a}{n}^T\nMat{b}{n},
\end{equation}
where we leave it as a short exercise for the reader to verify (e.g., by induction) that
\begin{equation}\label{eq:trMataMatb}
 \tr(\nMat{a}{n}^T\nMat{b}{n})=\skalarProd{\nMat{a}{n}}{\nMat{b}{n}}=(n-1)\,\skalarProd{a}{b}.
\end{equation}
Recall that the associated matrix $\Anti{.}$ with the usual cross product $\times$ in $\R^3$ is a (skew-symmetric) square matrix, while the associated matrix $\nMat{.}{n}$ with the generalized cross product $\ntimes{n}$ is an $(\frac{n(n-1)}{2}\times n)$-matrix and therefore is a square matrix only in the case of $n=3$. Hence, despite of the situation in Room's identity \eqref{eq:Room} we may also interchange the matrices in its $n$-dimensional analogue \eqref{eq:RoomN}, i.e.  consider the expression in \eqref{eq:acrossb_inRn}.

Returning to the usual \emph{Room identity} we have
\begin{subequations}
\begin{equation}
 \Anti{a}\times b = L(a\otimes b)  \quad \text{and} \quad a \otimes b = L(\Anti{a}\times b) \quad \forall \ a,b\in\R^3,\label{eq:RoomLin}
 \end{equation}
 denoting by $L(.)$ a corresponding linear operator with constant coefficients, not necessarily the same in any two places here and in the following.
 
On the one hand, we associate with the matrix $\Anti{.}$ a representation of the usual cross product. \emph{Room's identity} can be generalized to higher dimensions in three different ways. We have already seen in \eqref{eq:RoomN} and \eqref{eq:RoomNdyadic0} an extension to:
\begin{equation}\label{eq:RoomNMatMat}
 \nMat{a}{n}^T\nMat{b}{n}= L(a\otimes b)  \quad\text{and}\quad a \otimes b = L(\nMat{a}{n}^T\nMat{b}{n}) \quad \forall \ a,b\in\R^n.
\end{equation}
However, a similar result to \eqref{eq:RoomLin} also holds true for the generalized cross product of the matrix coming from the matrix representation of the generalized cross product with a vector, see \cite{agn_lewintan2020KornLpN_tracefree}:
\begin{equation}\label{eq:RoomNMat}
\begin{split}
 \nMat{a}{n}\ntimes{n} b &= L(a\otimes b)  \quad\forall \ a,b\in\R^n, n\ge2 \\
 \text{and}\quad
 a \otimes b &= L(\nMat{a}{n}\ntimes{n} b) \quad \forall \ a,b\in\R^n, n\ge3.
 \end{split}
\end{equation}
These relations also apply to the case of $a\ntimes{n}\nMat{b}{n}^T = \nMat{a}{n}\nMat{b}{n}^T=-\nMat{a}{n}\ntimes{n}b$, which for $n=2$ is only a scalar, so that the last relation in \eqref{eq:RoomNMat} is only valid for $n\ge3$.

On the other hand, \emph{Room's identity} in three dimensions can also be seen as an expression for the cross product of a skew-symmetric matrix with a vector:
\begin{equation}
 A\times b = L(\axl(A)\otimes b)  \quad\text{and}\quad \axl(A) \otimes b = L(A\times b) \quad\forall \ A\in\so3, b\in\R^3,\tag{\ref{eq:RoomLin}'}
 \end{equation}
where $\axl:\so{3}\to\R^3$ denotes the inverse of $\Anti{.}$. It is interesting that a similar result holds true  for $(n\times n)$-skew symmetric matrices in all dimensions $n\ge2$, see \cite{agn_lewintan2020KornLpN_tracefree}:
\begin{equation}
\begin{split}
 A\ntimes{n}b &= L(\a_n (A) \otimes b)  \\
 \text{and}\quad \a_n (A) \otimes b &= L(A\ntimes{n}b) \quad\forall \ A\in\so{n}, b\in\R^n,\label{eq:RoomNskew}
\end{split}
\end{equation}
\end{subequations}
where \eqref{eq:RoomNMat}$_1$ and \eqref{eq:RoomNskew}$_1$ follow directly from the definition of the generalized cross product of a matrix and a vector but for \eqref{eq:RoomNMat}$_2$ and \eqref{eq:RoomNskew}$_2$ inductive proofs are needed, cf.~\cite{agn_lewintan2020KornLpN_tracefree}.

\begin{remark}
 We have seen that \emph{Room's identity} \eqref{eq:Room} admits three different generalizations to higher dimensions \eqref{eq:RoomNMatMat}, \eqref{eq:RoomNMat}, \eqref{eq:RoomNskew} which coincide in three dimensions when considering the usual cross product and the associated matrix with it, since the latter is a skew-symmetric (square) matrix. Whereas, \emph{Grassmann's} and \emph{Jacobi's} identities generalize only in the ways presented in \eqref{eq:GrassmannN} and \eqref{eq:JacobiN}. Indeed, these relations are comparable to the situation in three dimensions when considering the usual triple vector product ~ $a\times(b\times c) = \Anti{a}(b\times c)$ ~ since ~ $\Anti{a}^T=-\Anti{a}$.
\end{remark}

\subsection{Simultaneous cross product}\label{sec:simultaneous}
Of special interest is a simultaneous cross product of a square matrix $P\in\R^{n\times n}$ and a vector $b\in\R^n$ from both sides:
\begin{equation}\label{eq:doublecrossN}
 b\ntimes{n}P\ntimes{n}b = \nMat{b}{n}P\nMat{-b}{n}^T \overset{\eqref{eq:crosslinkszurechts}}{=} -\left((P\ntimes{n}b)^T\ntimes{n}b\right)^T\in\R^{\frac{n(n-1)}{2}\times\frac{n(n-1)}{2}}\,,
\end{equation}
where, due to the associativity of matrix multiplication, we can omit parenthesis. Since
\begin{subequations}
\begin{equation}
 \left(b\ntimes{n}P\ntimes{n}b\right)^T\overset{\eqref{eq:doublecrossN}}{=} -\nMat{b}{n}P^T\nMat{-b}{n}^T= b\ntimes{n}P^T\ntimes{n}b
\end{equation}
it follows for $S\in\Sym{n}$ and $A\in\so{n}$ immediately:
\begin{equation}
  b\ntimes{n} S  \ntimes{n} b \in\Sym{\frac{n(n-1)}{2}} \quad \text{and}\quad b\ntimes{n} A \ntimes{n} b \in\so{\frac{n(n-1)}{2}} 
\end{equation}
 and for all $P\in\R^{n\times n}$:
 \begin{equation}\label{eq:doublecrossNsymmetry}
   b\ntimes{n} \sym{P} \ntimes{n} b = \sym(b\ntimes{n} P\ntimes{n} b),\quad  b\ntimes{n} \skew{P} \ntimes{n} b = \skew(b\ntimes{n} P\ntimes{n} b).
 \end{equation}
\end{subequations}
For $P=\id{n}$ the identity matrix we obtain
\begin{equation}
 b\ntimes{n}\id{n}\ntimes{n} b  = \nMat{b}{n}\nMat{-b}{n}^T \overset{\eqref{eq:bcrossb_inRn}}{=}
 \nMat{b}{n}\ntimes{n}b\in\Sym{\frac{n(n-1)}{2}}.
\end{equation}
Moreover, for $a,b,c\in\R^n$ it follows that
\begin{subequations}\label{eq:doublecrossdyadicN}
\begin{equation}
  b\ntimes{n} (a\otimes c)\ntimes{n} b  \overset{\eqref{eq:dyadiccrossN}}{=}(b\ntimes{n} a)\otimes (c\ntimes{n} b),
\end{equation}
and especially for $c=b$ that
\begin{align}
 b\ntimes{n} (a\otimes b)\ntimes{n} b &=b\ntimes{n} (b\otimes a)\ntimes{n} b\notag\\
 &= b\ntimes{n}\sym(a\otimes b)\ntimes{n} b = b\ntimes{n}\skew(a\otimes b)\ntimes{n} b=0 .
\end{align}
\end{subequations}
Furthermore, for a square matrix $\P\in\R^{\frac{n(n-1)}{2}\times\frac{n(n-1)}{2}}$ and a vector $b\in\R^n$ we obtain
\begin{equation}\label{eq:doubleb}
 \nMat{b}{n}^T\P\nMat{b}{n}\in\R^{n\times n},
\end{equation}
which has comparable properties to the simultaneous cross product above, for instance:
\begin{subequations}
\begin{align}
 \left(\nMat{b}{n}^T\P\nMat{b}{n}\right)^T&=\nMat{b}{n}^T\P^T\nMat{b}{n}
\intertext{which gives:}
  \sym \left(\nMat{b}{n}^T\P\nMat{b}{n}\right) &= \nMat{b}{n}^T\sym\P\nMat{b}{n}, \\
  \shortintertext{as well as}
  \skew \left(\nMat{b}{n}^T\P\nMat{b}{n}\right) &= \nMat{b}{n}^T\skew\P\nMat{b}{n}.
 \end{align}
\end{subequations}
And for the identity matrix $\P=\id{\frac{n(n-1)}{2}}$ we obtain:
\begin{equation}\label{eq:doublebId}
 \nMat{b}{n}^T\id{\frac{n(n-1)}{2}}\nMat{b}{n}= \nMat{b}{n}^T\nMat{b}{n} \overset{\eqref{eq:RoomNbb}}{=} \norm{b}^2\cdot\id{n}-b\otimes b.
\end{equation}
Again, the corresponding expressions to \eqref{eq:doublecrossN} and \eqref{eq:doubleb} coming from the usual cross product in three dimensions just coincide:
\begin{equation}
 b\times P \times b = \Anti{b} P\Anti{b}\in\R^{3\times3} \quad   \text{for } b\in\R^3, P\in\R^{3\times3}.
\end{equation}

\section{Differential operators}
Let us now come back to the interplay between a linear homogeneous differential operator with constant coefficients and its symbol, thus, replacing $b$ by the vector differential operator $\nabla$ in the algebraic relation presented in the previous sections. For that purpose, let $\Omega\subseteq\R^n$ be open, $n\ge2$ and $n,m\in\N$. As usual, the derivative and the divergence of a vector field rely on the dyadic product and the scalar product, respectively:
\begin{equation}
 \begin{split}
  \D a & \coloneqq a\otimes \nabla\in C^\infty_c(\Omega,\R^{m\times n}) \qquad \text{for } a\in C^\infty_c(\Omega,\R^m),\\
  \div a & \coloneqq \skalarProd{a}{\nabla}= a^T\nabla=\tr(\D a)\in C^\infty_c(\Omega,\R) \qquad \text{for } a\in C^\infty_c(\Omega,\R^n),
 \end{split}
\end{equation}
where the latter can be generalized to a matrix divergence in a \textit{row-wise} way:
\begin{equation}
 \Div P \coloneqq P\,\nabla\in C^\infty_c(\Omega,\R^{m}) \qquad \text{for } P\in C^\infty_c(\Omega,\R^{m\times n}).
\end{equation}
In three dimensions, the usual curl is seen as
\begin{equation}
 \curl a \coloneqq a \times (-\nabla) = \nabla \times a = \Anti{\nabla}a = 2\cdot\axl(\skew\D a)\quad  \text{for } a\in C^\infty_c(\Omega,\R^3), n=3.
\end{equation}
Similarly, in arbitrary dimension $n\ge2$ the generalized curl is related to the generalized cross product via
\begin{equation}
\begin{split}
 \ncurl{n}a &\coloneqq a\ntimes{n}(-\nabla)=\nabla\ntimes{n}a = \nMat{\nabla}{n}a \\
 &\overset{\eqref{eq:crossandskew}}{=} 2\cdot \a_n(\skew \D a)\in C^\infty_c(\Omega,\R^{\frac{n(n-1)}{2}})  \qquad \text{for } a\in C^\infty_c(\Omega,\R^n),
 \end{split}
\end{equation}
where the latter expression is usually considered in index-notation to introduce the generalized curl.

Furthermore, we consider the new differential operation
\begin{equation}\label{eq:newdiffop}
 \nMat{\nabla}{n}^T\a \in C^\infty_c(\Omega,\R^n) \quad \text{for } \a\in C^\infty_c(\Omega,\R^{\frac{n(n-1)}{2}}),
\end{equation}
which differs from the usual curl and from $\ncurl{\frac{n(n-1)}{2}}\a$ also in the three-dimensional case:
\begin{equation}
\curl \begin{pmatrix} \alpha_1\\\alpha_2\\\alpha_3 \end{pmatrix} =
 \begin{pmatrix}
  \partial_2\alpha_3 - \partial_3\alpha_2\\
  \partial_3\alpha_1 - \partial_1\alpha_3\\
  \partial_1\alpha_2 - \partial_2\alpha_1
 \end{pmatrix},
\quad
 \ncurl{3}\begin{pmatrix} \alpha_1\\\alpha_2\\\alpha_3 \end{pmatrix} =
 \begin{pmatrix}
  \partial_1\alpha_2 - \partial_2\alpha_1\\
  \partial_1\alpha_3 - \partial_3\alpha_1\\
  \partial_2\alpha_3 - \partial_3\alpha_2
 \end{pmatrix},
\quad
\nMat{\nabla}{3}^T\begin{pmatrix} \alpha_1\\\alpha_2\\\alpha_3 \end{pmatrix} =
 \begin{pmatrix}
  -\partial_2\alpha_1 - \partial_3\alpha_2\\
  \partial_1\alpha_1 - \partial_3\alpha_3\\
  \partial_1\alpha_2 + \partial_2\alpha_3
 \end{pmatrix}.
\end{equation}
To the best of our knowledge, the operator $\nMat{\nabla}{n}^T: C^\infty_c(\Omega,\R^{\frac{n(n-1)}{2}})\to C^\infty_c(\Omega,\R^n)$ has not received any attention in the literature so far, not even in index notation. However, this differential operator plays the counterpart in the integration by parts formula for the generalized $\ncurl{n}$, see \eqref{eq:partInt} below. This adjoint differential operator appears here because the matrix associated with the generalized cross product has no symmetry.

Furthermore, it is the matrix representations of the cross product which allows us to introduce also a \textit{row-wise} generalized matrix curl operator:
\begin{equation}
 \nCurl{n}P\coloneqq P\ntimes{n}(-\nabla)\overset{\eqref{eq:Matcrossprodright}}{=} P\nMat{\nabla}{n}^T  \qquad \text{for } P\in C^\infty_c(\Omega,\R^{m\times n}),
\end{equation}
which is connected to the the column-wise differential operation:
\begin{equation}
\nabla\ntimes{n}B\coloneqq\nMat{\nabla}{n}B\overset{\eqref{eq:crosslinkszurechts}}{=} \left[\nCurl{n}B^T\right]^T \qquad \text{for } B\in C^\infty_c(\Omega,\R^{n\times m}),
\end{equation}
and like in the three dimensional setting can be referred to as $\nCurl{n}^T$.

Moreover, the matrix representation of the curl operation offers also a further differential operator $(.)\nMat{\nabla}{n}$ for  $\left(m\times\frac{n(n-1)}{2}\right)$-matrix fields:
\begin{align}
 \P\nMat{\nabla}{n}\in C^\infty_c(\Omega,\R^{m\times n}) \quad \text{for } \P\in C^\infty_c(\Omega,\R^{m\times \frac{n(n-1)}{2}}),
\end{align}
i.e. the row-wise differentiation from \eqref{eq:newdiffop}, and again related by transposition also $\nMat{\nabla}{n}^T(.)$ for $\left(\frac{n(n-1)}{2}\times m\right)$-matrix fields.

Surely, it follows from \eqref{eq:dyadiccross2N}:
\begin{subequations}
 \begin{align}\label{eq:curlgrad}
  \ncurl{n}(\nabla \alpha) &\equiv 0 \qquad \text{for } \alpha\in C^\infty_c(\Omega,\R),
\intertext{or even}
\nCurl{n}(\D a) &\equiv 0 \qquad \text{for } a\in C^\infty_c(\Omega,\R^m),
 \end{align}
\end{subequations}
and from \eqref{eq:dyadiccross3N}:
\begin{align}
 \nCurl{n}(\D a)^T &= -2\cdot\nCurl{n}(\sym\D a)= 2\cdot\nCurl{n}(\skew\D a) \notag\\
 &= \left[\D\ncurl{n}a\right]^T = 2\cdot \left[\D\a_n(\skew\D a)\right]^T  \qquad \text{for } a\in C^\infty_c(\Omega,\R^n).
\end{align}
And, as analogue to the usual $\div\circ\curl\equiv 0$, we have in $n$-dimensions:
\begin{equation}\label{eq:KerDiv}
 \div \nMat{\nabla}{n}^T\a \overset{\eqref{eq:scalartripleNbb}}{\equiv}0 \qquad \text{for } \a\in C^\infty_c(\Omega,\R^{\frac{n(n-1)}{2}}).
\end{equation}

We recall the following definition.

\begin{definition} Let $\Omega\subseteq\R^n$ be open.
 A linear homogeneous differential operator with constant coefficients $\mathcal{A}:C^\infty_c(\Omega,\R^m)\to C^\infty_c(\Omega,\R^N)$ is said to be \emph{elliptic} if its symbol $\mathbb{A}(b)\in\operatorname{Lin}(\R^m,\R^N)$ is injective for all $b\in\R^n\backslash\{0\}$.
\end{definition}

It follows, from $b\times b=0$ for $b\in\R^3$ that the usual $\curl$ operator is not elliptic. Similarly, also the generalized $\ncurl{n}$ is not elliptic.

Since the kernel of $\begin{pmatrix}-\beta_2\\\beta_1\end{pmatrix}:\R\to\R^2$ consists only of $0$ for all $\begin{pmatrix}\beta_1\\\beta_2\end{pmatrix}\in\R^2\backslash\{0\}$, the operator $\nMat{\nabla}{2}^T$ is elliptic.

To see that $\nMat{\nabla}{n}^T$ is not elliptic for all $n\ge3$ we consider
\begin{equation}
 \nMat{\begin{pmatrix}\beta_1\\\beta_2\\\beta_3 \end{pmatrix}}{3}^T\begin{pmatrix}\beta_3\\-\beta_2\\\beta_1 \end{pmatrix}= \begin{pmatrix}-\beta_2 & -\beta_3 & 0 \\ \beta_1 & 0 & -\beta_3 \\ 0 & \beta_1 & \beta_2 \end{pmatrix}\begin{pmatrix}\beta_3\\-\beta_2\\\beta_1 \end{pmatrix}  =0
\end{equation}
which gives the non-ellipticity of  $\nMat{\nabla}{3}^T$ and the non-ellipticity in the higher dimensional cases follows from the inductive structure.

\subsection{Nye formulas}
Denoting by $\Curl$ the matrix curl operator related to the usual curl for vector fields in $\R^3$, \emph{Room's identity} \eqref{eq:RoomLin} becomes after interchanging $b$ by $\nabla$:
\begin{subequations}
\begin{align}\label{eq:NyeLin}
\Curl(\Anti{a})= L(\D a) \quad \text{and}\quad \D a = L(\Curl \Anti{a})  \quad \text{for } a\in C^\infty_c(\Omega,\R^3),
\end{align}
where $\Omega(\text{open})\subseteq\R^3$ for a moment. More precisely, they read
\begin{align}
\Curl(\Anti{a})&=\div a\cdot\id3-(\D a)^T \\
\shortintertext{and}
\D a &= \frac{\tr(\Curl \Anti{a})}{2}\cdot\id3-(\Curl \Anti{a})^T
\end{align}
and are better known as \emph{Nye formulas}  \cite[eq.\!\! (7)]{Nye53}. Surely, \eqref{eq:NyeLin}$_1$ is not surprising at all, but  \eqref{eq:NyeLin}$_2$ implies that the entries of the derivative of a skew-symmetric matrix field are linear combinations of the entries of the matrix curl:
\begin{align}\label{eq:derskew}
 \D A = L (\Curl A) \qquad \text{for } A\in C^\infty_c(\Omega,\so{3}).
\end{align}
\end{subequations}
Returning to the higher dimensional case we conclude, from \eqref{eq:RoomNdyadic0} or \eqref{eq:RoomNMatMat}, that
\begin{align}
 \D a = L(\nMat{a}{n}^T\nMat{\nabla}{n}) \qquad \text{for } a\in C^\infty_c(\Omega,\R^n),
\end{align}
and from \eqref{eq:RoomNMat}
\begin{align}
 \D a = L(\nCurl{n}\nMat{a}{n}) \qquad \text{for } a\in C^\infty_c(\Omega,\R^n), n\ge3.
\end{align}
Note, however, that the latter expression is (in general) not related to $\ncurl{n}a$.
Finally, from \eqref{eq:RoomNskew} we deduce that
\begin{align}
 \D \a_n(A) = L(\Curl_n A) \qquad \text{for } A\in C^\infty_c(\Omega,\so{n}).
\end{align}
which implies \eqref{eq:derskew} in all dimensions $n\ge2$:
\begin{align}
 \D A =  L (\nCurl{n} A) \qquad \text{for } A\in C^\infty_c(\Omega,\so{n}),
\end{align}
a relation that is usually derived in index notations.

\subsection{Incompatibility operator}
In three dimensions, the incompatibility $\inc$ is usually defined via
\begin{equation}
 \inc P \coloneqq \Curl[(\Curl P)^T] = -\nabla\times P^T\times \nabla \qquad \text{for } P\in C^\infty_c(\Omega,\R^{3\times 3}), n=3.
 \end{equation}
In higher dimensions, for $P\in C^\infty_c(\Omega,\R^{n\times n})$ we consider the generalized incompatibility operator (where for simplicity we drop the transposition and the minus sign) given by:
\begin{align}
 \inc_n P&\coloneqq\nabla\ntimes{n}P\ntimes{n}\nabla=-\nMat{\nabla}{n}P\nMat{\nabla}{n}^T\\
 &\overset{\mathclap{\eqref{eq:doublecrossN}}}{=} \ -\left[\nCurl{n}\left(\left(\nCurl{n}P\right)^T\right)\right]^T\in C^\infty_c(\Omega,\R^{\frac{n(n-1)}{2}\times \frac{n(n-1)}{2}}).
\end{align}
It possesses the properties known from the usual incompatibility operator in three dimensions, it follows namely from \eqref{eq:doublecrossNsymmetry} that
\begin{equation}
 \sym\inc_n P=\inc_n\sym P \quad \text{and}\quad \skew\inc_n P=\inc_n\skew P
\end{equation}
and from \eqref{eq:doublecrossdyadicN} for $a\in C^\infty_c(\Omega,\R^n)$:
\begin{align}
 \inc_n\D a=\inc_n(\D a)^T = \inc_n(\sym\D a) = \inc_n(\skew\D a)\equiv0.
\end{align}

Furthermore, for matrix fields $\P\in C^\infty_c(\Omega,\R^{\frac{n(n-1)}{2}\times \frac{n(n-1)}{2}})$ we consider the new differential operation
\begin{equation}
 \nMat{\nabla}{n}^T\P\nMat{\nabla}{n}\in C^\infty_c(\Omega,\R^{n\times n})
\end{equation}
with similar properties to the generalized incompatibility operator, see section \ref{sec:simultaneous}. Especially for $\zeta\in C^\infty_c(\Omega,\R)$ we obtain:
\begin{equation}
 \nMat{\nabla}{n}^T\zeta\cdot\id{\frac{n(n-1)}{2}}\nMat{\nabla}{n} \overset{\eqref{eq:doublebId}}{=} \Delta\zeta\cdot\id{n}-\D\nabla\zeta,
\end{equation}
where we have used that from an algebraic point of view the Laplacian $\Delta=\norm{\nabla}^2$ behaves like a scalar and where $\D\nabla\zeta$ is the Hessian matrix of $\zeta$. The latter expression reminds of the known identity in $n=3$ dimensions for the usual incompatibility operator:
\begin{equation}
 \inc (\zeta\cdot\id3) = \Delta\zeta\cdot\id3-\D\nabla\zeta.
\end{equation}
It is clear from the integration by parts formula for the generalized curl \eqref{eq:partIntMat}, how the operator $\nMat{\nabla}{n}^T(.)\nMat{\nabla}{n}$ 
plays the counterpart in the corresponding integration by parts formula for the generalized incompatibility operator. For the corresponding formula in the usual three dimensional case we refer to \cite{Amstutz2016analysis}.

\begin{remark}
 In three dimensions, the usual incompatibility operator $\inc$ occurs, e.g., in the modelling of dislocated crystals or in the modelling of elastic materials with dislocations, where the notion of incompatibility is at the basis of a new paradigm to describe the inelastic effects, see e.g.~\cite{Lazar2010disclocations,vanGoethem2011incompatibility,vanGoethem2012dislocations,Maggiani2015incompatible,Amstutz2016analysis,agn_ebobisse2017fourth}. The index-free view presented above should provide a better understanding of such phenomena also in higher dimensions.
\end{remark}

\subsection{Vector Laplacian}
 Recalling \eqref{eq:RoomNbb} we have for all $a,b\in\R^n$:
 \begin{equation}
  \nMat{b}{n}^Tb\ntimes{n}a=\nMat{b}{n}^T\nMat{b}{n}a \overset{\eqref{eq:RoomNbb}}{=}\norm{b}^2 \cdot a - b\cdot\skalarProd{b}{a}.
 \end{equation}
Thus, interchanging $b$ by $\nabla$ we deduce
\begin{equation}\label{eq:LaplacianN}
\Delta a = \nabla \div a + \nMat{\nabla}{n}^T\ncurl{n}a \qquad \text{for } a\in C^\infty_c(\Omega,\R^{n}), n\ge2,
\end{equation}
which is the generalization of the known expression for the vector Laplacian in $n=3$ dimensions:
\begin{equation}\label{eq:Laplacian3}
\Delta a = \nabla \div a - \curl\curl a \qquad \text{for } a\in C^\infty_c(\Omega,\R^{3}),
\end{equation}
and the appearance of the minus sign comes from the fact that the matrix associated with the usual cross product is a skew-symmetric matrix.

 Since the matrix divergence and matrix curl act row-wise, we obtain
 \begin{equation}
 \Delta P =  \D\Div P +(\nCurl{n}P)\nMat{\nabla}{n}\qquad \text{for } P\in C^\infty_c(\Omega,\R^{m\times n}),
\end{equation}
 for $m,n\in\N$, $n\ge2$, meaning that the entries of the Laplacian of a matrix field $P$ are linear combinations of the entries of the derivative of the matrix curl and of the entries of the derivative of the matrix divergence.

\subsection{Integration by parts}
For the sake of completeness, we include the integration by parts formula for the generalized matrix curl:
Let $\Omega\subset\R^n$ be an open and bounded set with Lipschitz boundary $\partial\Omega$ and outward unit normal $\nu$. For all $a\in C^1(\overline{\Omega},\R^{n})$ and all $\a\in C^1(\overline{\Omega},\R^{\frac{n(n-1)}{2}})$ we have
\begin{subequations}
\begin{align}\label{eq:partInt}
 \int_{\Omega} \skalarProd{\ncurl{n}a}{\a}+\skalarProd{a}{\nMat{\nabla}{n}^T\a}\, \mathrm{d}x &= \int_{\partial\Omega}\skalarProd{a\ntimes{n}(-\nu)}{\a}\,\mathrm{d}S,
\intertext{so that for matrix fields $P\in C^1(\overline{\Omega},\R^{m\times n})$ and  $\P\in C^1(\overline{\Omega},\R^{m\times \frac{n(n-1)}{2}})$ it follows}
\label{eq:partIntMat}
 \int_{\Omega} \skalarProd{\nCurl{n}P}{\P}+\skalarProd{P}{\P\nMat{\nabla}{n}}\, \mathrm{d}x &= \int_{\partial\Omega}\skalarProd{P\ntimes{n}(-\nu)}{\P}\,\mathrm{d}S,
\end{align}
\end{subequations}
and we refer to \cite{agn_lewintan2020KornLpN_tracefree} for a coordinate-free proof for square matrix fields $P$.

\subsection{Helmholtz decomposition}
It is well known that any vector field $a\in C^\infty_c(\R^n,\R^n)$ admits a decomposition into a divergence-free vector field and a gradient field, i.e.~a $\ncurl{n}$-free part, see e.g.~\cite{FujiwaraMorimoto} and for a deviation from the Hodge decomposition see \cite{Iwaniec1992}. Let us denote the divergence-free part by $a_{\div}$ and the $\ncurl{n}$-free by $a_{\ncurl{n}}$, so having $a=a_{\div}+a_{\ncurl{n}}$. At the end of our vector calculus, we give the explicit index-free expressions of these parts and provide the Helmholtz decomposition explicitly in all dimensions $n\ge2$. More precisely, we show that
\begin{subequations}\label{eq:HelmholtzN}
\begin{align}
 a_{\ncurl{n}}(x) &= \nabla_x\int_{\R^n} \G\cdot\div a(y)\,\mathrm{d}y\\
 &=\frac{1}{n\omega_n} \int_{\R^n} (x-y)\cdot\left(\norm{x-y}^{-n}\div a(y)\right)\,\mathrm{d}y, \\
 \intertext{and}
 a_{\div}(x)&= \nMat{\nabla_x}{n}^T\int_{\R^n} \G\cdot\ncurl{n}a(y)\,\mathrm{d}y\\
 &=\frac{1}{n\omega_n}\int_{\R^n} \nMat{x-y}{n}^T\left(\norm{x-y}^{-n}\cdot\ncurl{n}a(y)\right)\,\mathrm{d}y,
\end{align}
\end{subequations}
where $\G$ denotes the normalized \textit{fundamental Green function for the Laplacian} for the entire space $\R^n$ and is given by
\begin{equation}
 \G=\begin{cases}
     \frac{1}{2\pi}\ln\norm{x-y}, & \text{for } n=2, \\
     \frac{1}{n(2-n)\omega_n}\norm{x-y}^{2-n}, &\text{for } n\ge3,
    \end{cases}
\end{equation}
denoting by $\omega_n$ the volume of the unit ball in $\R^n$, see \cite[Section 2.4]{GT}. Indeed, the first expressions in \eqref{eq:HelmholtzN} follow from the decomposition of the vector Laplacian in \eqref{eq:LaplacianN} since for $a\in C^\infty_c(\R^n,\R^n)$ we have
\begin{align}
 a(x)&=\int_{\R^n}a(y)\cdot\Delta_x\G\,\mathrm{d}y = \Delta_x \int_{\R^n}a(y)\cdot\G \,\mathrm{d}y \notag\\
  &\overset{\mathclap{\eqref{eq:LaplacianN}}}{=} \ \nabla_x\div_x\int_{\R^n}a(y)\cdot\G \,\mathrm{d}y + \nMat{\nabla_x}{n}^T\ncurl{n}_{,x}\int_{\R^n}a(y)\cdot\G \,\mathrm{d}y \notag\\
  & = \nabla_x\int_{\R^n}\skalarProd{a(y)}{\nabla_x\G}\,\mathrm{d}y+ \nMat{\nabla_x}{n}^T\int_{\R^n}\nabla_x\G\ntimes{n} a(y)\,\mathrm{d}y\notag\\
  & \overset{(\ast)}{=}\nabla_x \int_{\R^n}  \skalarProd{a(y)}{-\nabla_y\G}\,\mathrm{d}y+\nMat{\nabla_x}{n}^T\int_{\R^n}a(y)\ntimes{n}\nabla_y\G\,\mathrm{d}y\notag\\
  & \overset{(\ast\ast)}{=} \nabla_x\int_{\R^n} \G\cdot\div a(y)\,\mathrm{d}y + \nMat{\nabla_x}{n}^T\int_{\R^n} \G\cdot\ncurl{n}a(y)\,\mathrm{d}y\,,\notag
\end{align}
where in $(\ast)$ we used that $\nabla_x\G = -\nabla_y\G$ and in $(\ast\ast)$ the relations
\begin{align}\label{eq:product}
\div(\alpha\cdot a)=\skalarProd{\nabla\alpha}{a} +\alpha\div a \quad\text{and}\quad
\ncurl{n}(\alpha\cdot a)=\nabla\alpha\ntimes{n}a+\alpha\cdot\ncurl{n}a,
\end{align}
for $\alpha\in C^\infty_c(\Omega)$ and $a\in C^\infty_c(\Omega,\R^n)$.  Since we have 
  \begin{equation}
 \nabla_x\G=\frac{1}{n\omega_n}\norm{x-y}^{-n}\cdot(x-y) \qquad \text{for } n\ge2,
 \end{equation}
we obtain
\begin{subequations}
\begin{align}
 a_{\ncurl{n}}(x) &=\frac{1}{n\omega_n} \int_{\R^n} (x-y)\cdot\left(\norm{x-y}^{-n}\div a(y)\right)\,\mathrm{d}y, \\
 a_{\div}(x)&= \frac{1}{n\omega_n}\int_{\R^n} \nMat{x-y}{n}^T\left(\norm{x-y}^{-n}\cdot\ncurl{n}a(y)\right)\,\mathrm{d}y,
\end{align}
\end{subequations}
and end up with \textit{Riesz potentials} of order $1$, see \cite[Section V.1]{Stein}.

\subsection{Robbin's proof of the div-curl lemma in higher dimensions}
In this last section, we show that the proof of the div-curl lemma presented in \cite{Robbin} in  three dimensions can be directly adopted to all dimensions using the matrix representation of the generalized cross product presented above. More precisely, we show
 
\begin{lemma}\label{lem:divcurlN}
  Let $n\ge2$, $\Omega\subseteq\R^n$ be open and the sequences of functions $u^k, v^k:\Omega\to\R^n$ satisfy
  \begin{subequations}
  \begin{align}\label{eq:uk}
   u^k\rightharpoonup u \quad &\text{ in } L^2(\Omega,\R^n) \text{ for } k\to\infty,\\
   \label{eq:vk}
   v^k\rightharpoonup v \quad &\text{ in } L^2(\Omega,\R^n)\text{ for } k\to\infty,
  \end{align}
  \end{subequations}
   and
  \begin{subequations}
  \begin{align}\label{eq:div1}
   \{\div u^k\}_{k\in\N} \quad&\text{compact set in } H^{-1}_{\text{loc}}(\Omega,\R^n),\\
   \label{eq:curl1}
   \{\ncurl{n} v^k\}_{k\in\N} \quad&\text{compact set in } H^{-1}_{\text{loc}}(\Omega,\R^{\frac{n(n-1)}{2}}).
  \end{align}
  \end{subequations}
  Then in the sense of distributions we have
\begin{align}
 \langle u^k,v^k\rangle&\rightarrow \skalarProd{u}{v} \quad\text{ for } k\to\infty,
 \intertext{i.e. it holds for all $\varphi\in C^\infty_c(\Omega,\R)$:}
 \int_{\Omega}\varphi\,\langle u^k,v^k\rangle\,\mathrm{d}x&\rightarrow \int_{\Omega}\varphi\,\skalarProd{u}{v}\,\mathrm{d}x\quad \text{ for } k\to\infty.\notag
\end{align}
\end{lemma}
It is not the most general formulation of the div-curl lemma, and we refer to \cite{BCDM2009,Tartar2009} and the references contained therein for both historical comments and generalizations. The main objective here is to demonstrate that the algebraic view advocated in the previous sections allows us to carry out the proof from \cite{Robbin} even in all dimensions without introducing the language of differential geometry. In three dimensions, Robbin's proof is based on the decomposition of the vector Laplacian \eqref{eq:Laplacian3}. In the previous section, we have obtained the desired decomposition in all dimensions, see \eqref{eq:LaplacianN}.

Furthermore, multiplying \eqref{eq:RoomNbb} with $\nMat{b}{n}$ from the left we deduce for all $b\in\R^n$:
\begin{equation}
 \nMat{b}{n}\nMat{b}{n}^T\nMat{b}{n}=\norm{b}^2\cdot\nMat{b}{n}-(\nMat{b}{n}b)\otimes b = \norm{b}^2\cdot\nMat{b}{n}-(\underset{=0}{\underbrace{b\ntimes{n}b}})\otimes b =\norm{b}^2\cdot\nMat{b}{n}
\end{equation}
so that in the language of vector calculus it becomes:
\begin{equation}\label{eq:curlcurlcurl}
 \ncurl{n}\nMat{\nabla}{n}^T\ncurl{n} a = \Delta \ncurl{n} a \quad \text{ for }a\in C^\infty_c(\Omega,\R^n).
\end{equation}
Now we have prepared all the relations between differential operators that we need  to follow Robbin's proof.
\begin{proof}[Proof of Lemma \ref{lem:divcurlN}]
  It suffices to assume that the functions $u^k$ and $v^k$ have compact support, see \cite{Robbin} and the corresponding relations \eqref{eq:product}. Let us extend $u^k$ by zero to the entire space $\R^n$ and denote by $w^k$ the unique solution in $L^2(\Omega,\R^n)$ of
  \begin{equation}\label{eq:solLap_k}
   \Delta w^k=u^k.
  \end{equation}
Thus, we write
\begin{equation}
 w^k = \Delta^{-1}u^k
\end{equation}
and set
\begin{equation}
 \psi^k = \div w^k \quad\text{and}\quad \mathfrak{g}^k=\ncurl{n}w^k,
\end{equation}
so that
\begin{equation}\label{eq:psig}
 \nabla \psi^k+\nMat{\nabla}{n}^T\mathfrak{g}^k= \nabla\div w^k+\nMat{\nabla}{n}^T\ncurl{n}w^k\overset{\eqref{eq:LaplacianN}}{=}\Delta w^k \overset{\eqref{eq:solLap_k}}{=}u^k.
\end{equation}
Next, we show that $\Delta^{-1}$ commutes with the $\div$ and $\ncurl{n}$ operators in all dimensions $n\ge2$:
\begin{align}
 \Delta^{-1}\div f & = \Delta^{-1}\div\Delta\Delta^{-1}f \overset{\eqref{eq:LaplacianN}}{=} \Delta^{-1}\div(\nabla \div +\nMat{\nabla}{n}^T\ncurl{n})\Delta^{-1}f\notag\\
 &\overset{\eqref{eq:KerDiv}}{=}\Delta^{-1}\div\nabla \div \Delta^{-1}f= \Delta^{-1}\Delta \div \Delta^{-1}f = \div \Delta^{-1}f
\end{align}
as well as
\begin{align}
 \Delta^{-1}\ncurl{n}f & = \Delta^{-1}\ncurl{n}\Delta\Delta^{-1}f \overset{\eqref{eq:LaplacianN}}{=} \Delta^{-1}\ncurl{n}(\nabla \div +\nMat{\nabla}{n}^T\ncurl{n})\Delta^{-1}f\notag\\
 &\overset{\eqref{eq:curlgrad}}{=}\Delta^{-1}\ncurl{n}\nMat{\nabla}{n}^T\ncurl{n}\Delta^{-1}f\overset{\eqref{eq:curlcurlcurl}}{=} \Delta^{-1}\Delta \ncurl{n} \Delta^{-1}f =\ncurl{n} \Delta^{-1}f.
\end{align}
And we can conclude as in \cite{Robbin}:
\begin{align}
 \nabla \psi^k =\nabla\div w^k=\nabla\div\Delta^{-1}u^k=\nabla\Delta^{-1}\div u^k \to \nabla \psi \text{ in } L^2(\Omega,\R^n)
\end{align}
at least for a subsequence by \eqref{eq:div1}. Moreover,
\begin{align}
 \nMat{\nabla}{n}^T\underset{\eqqcolon \mathfrak{f}^k}{\underbrace{\ncurl{n}\Delta^{-1}v^k}}=\nMat{\nabla}{n}^T\Delta^{-1}\ncurl{n}v^k \to \nMat{\nabla}{n}^T \mathfrak{f} \text{ in } L^2(\Omega,\R^n)
\end{align}
at least for a subsequence by \eqref{eq:curl1}. Furthermore:
\begin{align}
 \mathfrak{g}^k %=\ncurl{n}w^k=\ncurl{n}\Delta^{-1}u^k
 \to \mathfrak{g}  \text{ in } L^2(\Omega,\R^{\frac{n(n-1)}{2}}), 
\quad \nMat{\nabla}{n}^T\mathfrak{g}^k%\overset{\eqref{eq:psig}}{=}u^k-\nabla \psi^k\rightharpoonup  u - \nabla \psi\text{ in } L^2(\Omega,\R^n)
\rightharpoonup\nMat{\nabla}{n}^T\mathfrak{g}\text{ in } L^2(\Omega,\R^n)
\end{align}
and for $\phi^k\coloneqq\div\Delta^{-1}v^k$:
\begin{align} 
 \phi^k\to\phi \text{ in } L^2(\Omega,\R), \quad \nabla \phi^k\rightharpoonup\nabla \phi \text{ in } L^2(\Omega,\R^n).
\end{align}
With the above results the proof completes as in \cite{Robbin}.
 \end{proof}

\section{Conclusion}
In the present paper, we have studied the algebraic structures underlying the generalized cross product, by relating it to an adequate matrix multiplication. The situation differs from that of the usual cross product in three dimensions, where a matrix representation results in a skew-symmetric matrix. The lack of symmetry in the general case leads to the fact that the known algebraic identities have to be adapted in an appropriate way and that also other combinations must be included. In vector calculus, this led not only to the generalized $\ncurl{n}$, but also to a new operator $\nMat{\nabla}{n}^T$. The importance of the latter has been highlighted in the previous sections, in particular by the fact that the image of the $\nMat{\nabla}{n}^T$ operator lies in the kernel of the divergence operator, see \eqref{eq:KerDiv}. Here we have thoroughly examined the matrix analysis behind such operations. Such a view has already proved very useful in extending Korn inequalities for
incompatible tensor fields to higher dimensions, cf.~\cite{agn_lewintan2020KornLpN_tracefree}, where first results in these matrix representations
have been obtained. With the better understanding presented here, we are now in a position to further extend Korn-Maxwell-Sobolev type inequalities. This will be the subject of a forthcoming paper.

\subsection*{Acknowledgment}
The author was supported by the Deutsche Forschungsgemeinschaft (Project-ID 415894848) and thanks the anonymous referees for their valuable comments and suggestions.

\footnotesize{
\printbibliography}

\end{document}